\documentclass[a4paper,12pt]{amsart}
\usepackage{amsmath,amssymb}

\usepackage{booktabs}

\usepackage{amsmath,amsthm,amsfonts,amssymb,amscd}
\usepackage{fullpage}
\usepackage{lastpage}
\usepackage{enumerate}
\usepackage{fancyhdr}

\usepackage{mathrsfs}
\usepackage[unicode]{hyperref}
\usepackage{latexsym,epsfig}          %allows figures to be inserted in .eps format

\usepackage{enumitem}% enumerate Roman
\usepackage{palatino}

\usepackage{tikz-cd}

\newtheorem{thm}{Theorem}%[section]

\newtheorem{remark}[thm]{Remark}

\newtheorem{prop}[thm]{Proposition}

\newtheorem{coro}[thm]{Corollary}

\def\dim{\operatorname{dim}}

\DeclareMathOperator{\Ext}{Ext}
\DeclareMathOperator{\Pic}{Pic}
\DeclareMathOperator{\Sym}{Sym}

\title{ Decomposition of the Tschirnhausen module for coverings on decomposable \texorpdfstring{$\mathbb{P}^1$}{P1}-bundles }

\author{Youngook Choi}
\address{Department of Mathematics Education, Yeungnam University, 280 Daehak-Ro, \hfill \newline\texttt{}
 \indent Gyeongsan, Gyeongbuk 38541, Republic of Korea}
\email{ychoi824@yu.ac.kr}

\author[H. Iliev]{Hristo Iliev}
\address{American University in Bulgaria, 2700 Blagoevgrad, Bulgaria, and \hfill \newline\texttt{}
 \indent Institute of Mathematics and Informatics, Bulgarian Academy of Sciences, \hfill \newline\texttt{}
 \indent 1113 Sofia, Bulgaria}
\email{ hiliev@aubg.edu, hki@math.bas.bg}

\author[S. Kim]{Seonja Kim}
\address{Department of Electronic Engineering, Chungwoon University, Sukgol-ro, Nam-gu, \hfill \newline\texttt{}
 \indent Incheon 22100, Republic of Korea}
\email{sjkim@chungwoon.ac.kr}

\thanks{The first author was supported by the National Research Foundation of Korea (NRF) grant funded by the Korea government(MSIT) (RS-2024-00352592).
The second author was supported by Grant KP-06-N 62/5 of Bulgarian National Science Fund.
The third  author was supported by the National Research Foundation of Korea (NRF) grant funded by the Korea government(MSIT) (2022R1A2C1005977).}

\usepackage{setspace}
\usepackage{geometry}
\begin{document}

\begin{abstract}
In this note, we show that for a smooth algebraic variety $Y$ and a smooth $m$-section $X$ of the $\mathbb{P}^1$-bundle
\[
f : \mathbb{P}(\mathcal{O}_Y \oplus \mathcal{O}_Y(E)) \longrightarrow Y,
\]
where $E$ is an effective divisor on $Y$ satisfying $H^1(Y, \mathcal{O}_Y(kE)) = 0$ for all $k = 1, \ldots, m-1$, the Tschirnhausen module of the induced covering
$
 f|_X : X \longrightarrow Y
$
is completely decomposable. We then apply it to coverings of curves arising in such a way.
\end{abstract}

\maketitle



\section{Introduction and statement}\label{Sec1}

Let $\varphi : X \to Y$ be a finite morphism of degree $m \geq 2$, where $X$ and $Y$ are smooth projective algebraic varieties of dimension $n \geq 1$. Such a covering induces the following short exact sequence of vector bundles on $Y$:
\begin{equation*}\label{Sec1_SES_Tschirn_Mod}
 0 \to \mathcal{O}_Y \xrightarrow{\varphi^{\sharp}} \varphi_{\ast} \mathcal{O}_X \to \mathcal{E}^{\vee} \to 0 \, ,
\end{equation*}
where $\mathcal{E}^{\vee}$ is known as the \emph{Tschirnhausen module} associated with the cover~$\varphi$. The theory of $m:1$  covers in algebraic geometry was developed by Miranda \cite{Mir85} in the case $m=3$, and by Casnati and Ekedahl in general, see \cite{CE96}. Partial results, examples, and applications have since been obtained by many others, including \cite{Kan13}, which described the resolution of $\mathcal{O}_X$ for $m=3,4,5$ in the case of curves, and \cite{DP22}, which showed that every vector bundle on a smooth projective curve arises, up to twist, as a Tschirnhausen module.

In this work, we focus on the situation where the covering arises from a decomposable $\mathbb{P}^1$-bundle over the base variety. The main result of the paper is the following theorem.

\begin{thm}\label{TheoremM}
Let $Y$ be a smooth projective variety of dimension $n \geq 1$, and let $E$ and $\Delta$ be effective divisors on $Y$ such that
\[
    H^1(Y, \mathcal{O}_Y(kE+\Delta)) = 0 \quad \text{ for } k = 1, 2, \ldots, m-1 ,
\]
where $m \geq 2$. Consider the projective bundle
$
    \mathcal{B} := \mathbb{P}(\mathcal{O}_Y \oplus \mathcal{O}_Y(E)),
$
following Hartshorne's convention, with projection
$
    f : \mathcal{B} \to Y.
$
Let $H$ denote the tautological divisor on $\mathcal{B}$, so that
\(\mathcal{O}_{\mathcal{B}}(H) \cong \mathcal{O}_{\mathcal{B}}(1)\),
and observe that
$
   f_{\ast} \mathcal{O}_{\mathcal{B}} (1)
   \;\cong\;
   \mathcal{O}_Y \oplus \mathcal{O}_Y (E) .
$
Suppose that
\[
    X \;\in\; \big| mH + f^{\ast}\Delta \big|
\]
is a smooth, irreducible divisor (an $m$-section of $f$), and let
$
    \varphi := f|_X : X \longrightarrow Y
$
be the restriction of the projection $f$ to $X$. Then:
\begin{enumerate}[label=(\alph*), leftmargin=*, font=\rmfamily]
 \item The pushforward of the structure sheaf of $X$ decomposes as
\[
    \varphi_{\ast} \mathcal{O}_X
    \;\cong\;
    \mathcal{O}_Y
    \;\oplus\;
    \mathcal{O}_Y (-E-\Delta)
    \;\oplus\;
    \cdots
    \;\oplus\;
    \mathcal{O}_Y \bigl(-(m-1)E-\Delta \bigr) .
\]
In particular, the \emph{Tschirnhausen module} $\mathcal{E}^{\vee}$ associated to the covering
$\varphi : X \to Y$ satisfies
\[
 \mathcal{E}^{\vee}
 \;\cong\;
 \left[
 \mathcal{O}_Y
 \;\oplus\;
 \mathcal{O}_Y (-E)
 \;\oplus\;
 \cdots
 \;\oplus\;
 \mathcal{O}_Y (-(m-2)E)
 \right]
 \otimes
 \mathcal{O}_Y (-E-\Delta) .
\]
 \item For the twisted sheaf $\mathcal{O}_X \bigl((H - f^{\ast}E)|_X \bigr)$, one has
\[
    \varphi_{\ast}
    \mathcal{O}_X \!\big( (H - f^{\ast}E)|_X \big)
    \;\cong\;
    \mathcal{O}_Y
    \;\oplus\;
    \mathcal{O}_Y (-E)
    \;\oplus\;
    \mathcal{O}_Y (-2E-\Delta)
    \;\oplus\;
    \cdots
    \;\oplus\;
    \mathcal{O}_Y \bigl( -(m-1)E-\Delta \bigr) .
\]
\end{enumerate}
\end{thm}

We discovered this result while studying the construction of components of the Hilbert scheme of curves via $m:1$ covers. The case $m = 3$ with $\Delta = 0$, was previously obtained by Fujita in \cite{Fuj1988}. In our earlier work \cite{CIK24b}, we considered the case $m = 3$, $\dim Y = 1$, and $\Delta$ equal to a point, motivated by the need to determine the dimension of the space of first-order deformations of $X$. Subsequently, we realized that the decomposition extends to the more general setting described in \textnormal{Theorem}~\ref{TheoremM}. This result can be applied in the construction of components of the Hilbert scheme of curves, as detailed in \textnormal{Theorem}~$A_m$ and \textnormal{Theorem}~$B_m$ of \cite{CIK25}.

The proof of \textnormal{Theorem~\ref{TheoremM}} is given in \textnormal{Section~\ref{Sec2}}, while in \textnormal{Section~\ref{Sec3}} we derive several consequences, some of which have applications to the study of the Hilbert scheme of curves, as can be seen, for instance, in \cite{CIK25}.

\subsection*{Conventions and notation}

Throughout the paper we work over the field $\mathbb{C}$. By \emph{smooth projective variety} we mean a smooth, integral, projective scheme of finite type over $\mathbb{C}$. For a vector bundle $\mathcal{V}$ on a variety $Y$, we adopt Hartshorne's convention for the projectivization $\mathbb{P} (\mathcal{V})$. For additional definitions and results not introduced explicitly in the paper, we refer the reader to \cite{Hart77}.

\medskip

\section{Proof of Theorem~\ref{TheoremM}}\label{Sec2}

The main tool in the proof is the following proposition.

\begin{prop}\label{PropDerivedDecomp}
Let $Y$ be a smooth projective variety of dimension $n$, and let $\mathcal{E}$ be a vector bundle of rank $r+1$ on $Y$. Consider the projective bundle
\[
 \mathcal{B} := \mathbb{P}(\mathcal{E})
\]
in the sense of Hartshorne's convention, so that $\mathcal{B}$ is a smooth variety of dimension $n + r$, and the fibers of the projection
\[
 f : \mathcal{B} \to Y
\]
are projective spaces of dimension $r$. Let $H$ denote the tautological divisor on $\mathcal{B}$, i.e., $\mathcal{O}_{\mathcal{B}} (H) \cong \mathcal{O}_{\mathcal{B}} (1)$, with $f_{\ast} \mathcal{O}_{\mathcal{B}} (H) \cong \mathcal{E}$. Then:

\begin{enumerate}[label=(\alph*), leftmargin=*]
  \item The Picard group of $\mathcal{B}$ is given by
  \[
  \operatorname{\Pic}(\mathcal{B}) \cong \mathbb{Z}[\mathcal{O}_{\mathcal{B}} (H)] \oplus f^{\ast} \operatorname{\Pic}(Y).
  \]

  \item The higher direct images of $\mathcal{O}_{\mathcal{B}} (k) := \mathcal{O}_{\mathcal{B}} (kH)$ satisfy:
  \begin{enumerate}[label=(\roman*)]
    \item For all $k \in \mathbb{Z}$,
    \[
    f_{\ast} \mathcal{O}_{\mathcal{B}} (k) \cong
    \begin{cases}
      \operatorname{\Sym}^k \mathcal{E} & \text{ if } k \geq 0, \\
      0 & \text{if } k < 0.
    \end{cases}
    \]

    \item For all $k \in \mathbb{Z}$ and $0 < i < r$,
    \[
    R^i f_{\ast} \mathcal{O}_{\mathcal{B}} (k) = 0.
    \]

    \item For $i = r$,
    \[
    R^r f_{\ast} \mathcal{O}_{\mathcal{B}} (k) \cong
    \begin{cases}
      0 & \text{if } k > -r - 1, \\
      \bigl(f_{\ast} \mathcal{O}_{\mathcal{B}} (-k - r - 1)\bigr)^{\vee} \otimes \det \mathcal{E}^{\vee} & \text{ if } k \leq -r - 1.
    \end{cases}
    \]
  \end{enumerate}
\end{enumerate}
\end{prop}
\begin{proof}
See \cite[Exercise~8.4, p.~252]{Hart77}.
\end{proof}

We also recall the \emph{projection formula} \cite[Exercise~8.3, p.~252]{Hart77}, which will be applied repeatedly throughout the proof, without explicit reference. Specifically, if $f : M \to N$ is a morphism of ringed spaces, $\mathcal{F}$ is an $\mathcal{O}_M$-module, and $\mathcal{V}$ is a locally free $\mathcal{O}_N$-module of finite rank, then
\[
 R^i f_{\ast} (\mathcal{F} \otimes f^{\ast} \mathcal{V})
 \;\cong\;
 R^i f_{\ast}(\mathcal{F}) \otimes \mathcal{V}.
\]

We are now ready to proceed with the proof of the theorem.

\begin{proof}[Proof of Theorem \ref{TheoremM}]

First, we give the proof of \textnormal{(a)}.

Let $X \in |mH+f^{\ast}\Delta|$ be a smooth variety. Apply $f_{\ast}$ to the exact sequence
\[
 0 \to \mathcal{O}_{\mathcal{B}} (-mH - f^{\ast}\Delta)  \to \mathcal{O}_{\mathcal{B}} \to \mathcal{O}_X \to 0 .
\]
Since $m \geq 1$ and $\Delta$ is effective, it follows from \textnormal{Proposition \ref{PropDerivedDecomp}} that
\[
 f_{\ast} \mathcal{O}_{\mathcal{B}} (-mH-f^{\ast}\Delta) = 0, \quad f_{\ast} \mathcal{O}_{\mathcal{B}} = \mathcal{O}_Y, \quad \mbox{ and } \quad R^1 f_{\ast} \mathcal{O}_{\mathcal{B}} = 0 .
\]
Thus, we obtain the exact sequence
\begin{equation}\label{ExSeq1}
 0 \to \mathcal{O}_Y \to \varphi_{\ast} \mathcal{O}_X \to R^1 f_{\ast} \mathcal{O}_{\mathcal{B}} (-mH-f^{\ast}\Delta) \to 0 .
\end{equation}
Applying \textnormal{Proposition \ref{PropDerivedDecomp}} again, we find
\[
\begin{aligned}
 R^1 f_{\ast} \mathcal{O}_{\mathcal{B}} (-mH-f^{\ast}\Delta)
 & \cong R^1 f_{\ast} \mathcal{O}_{\mathcal{B}} (-mH) \otimes \mathcal{O}_Y (-\Delta) \\
 & \cong \bigl(f_{\ast} \mathcal{O}_{\mathcal{B}} (m - 2)\bigr)^{\vee} \otimes \det \bigl(\mathcal{O}_Y \oplus \mathcal{O}_Y (E)\bigr)^{\vee} \otimes \mathcal{O}_Y (-\Delta) \\
 & \cong \bigl(\Sym^{m-2} (\mathcal{O}_Y \oplus \mathcal{O}_Y (E))\bigr)^{\vee} \otimes \mathcal{O}_Y (-E-\Delta) \\
 & \cong \mathcal{O}_Y (-E-\Delta) \oplus \cdots \oplus \mathcal{O}_Y \bigl(-(m-1)E-\Delta\bigr) \, .
\end{aligned}
\]
Since, by assumption,
\[
 h^1 \bigl(Y, \mathcal{O}_Y (k E+\Delta)\bigr) = 0 \quad \mbox{ for } \quad k = 1, \ldots , m-1,
\]
we have
\[
 \Ext^1 \bigl(R^1 f_{\ast} \mathcal{O}_{\mathcal{B}} (-mH-f^{\ast}\Delta), \mathcal{O}_Y\bigr) \cong H^1 \bigl(Y, \mathcal{O}_Y (E+\Delta) \oplus \cdots \oplus \mathcal{O}_Y ((m-1)E+\Delta)\bigr) = 0 .
\]
Therefore, the exact sequence~\eqref{ExSeq1} splits, and we conclude that
\[
 \varphi_{\ast} \mathcal{O}_X \cong \mathcal{O}_Y \oplus \mathcal{O}_Y (-E-\Delta) \oplus \cdots \oplus \mathcal{O}_Y \bigl(-(m-1)E-\Delta \bigr) .
\]

The proof of \textnormal{(b)} proceeds in a similar way.

Consider the exact sequence
\[
 0 \to \mathcal{O}_{\mathcal{B}} \bigl(-(m-1)H -f^{\ast}(E+\Delta)\bigr)  \to \mathcal{O}_{\mathcal{B}} (H-f^{\ast}E) \to \mathcal{O}_X \bigl((H-f^{\ast}E)_{|_X}\bigr) \to 0 ,
\]
apply $f_{\ast}$. Since $m \geq 2$,  \textnormal{Proposition \ref{PropDerivedDecomp}} yields:
\begin{itemize}[label=$\circ$, leftmargin=1cm, font=\rmfamily]
 \item $f_{\ast} \mathcal{O}_{\mathcal{B}} \bigl(-(m-1)H -f^{\ast}(E+\Delta)\bigr) \cong f_{\ast} \mathcal{O}_{\mathcal{B}} \bigl(-(m-1)H\bigr) \otimes \mathcal{O}_Y (-E-\Delta) = 0$,
 \item $f_{\ast} \mathcal{O}_{\mathcal{B}} (H-f^{\ast}E) \cong f_{\ast} \mathcal{O}_{\mathcal{B}} (H) \otimes \mathcal{O}_Y (-E) \cong \mathcal{O}_Y (-E) \oplus \mathcal{O}_Y$, and
 \item $R^1 f_{\ast} \mathcal{O}_{\mathcal{B}} (H-f^{\ast}E) \cong R^1 f_{\ast} \mathcal{O}_{\mathcal{B}} (1) \otimes \mathcal{O}_Y (-E) = 0$.
\end{itemize}
Thus, we obtain the exact sequence
\begin{equation}\label{ExSeq2}
 0 \to \mathcal{O}_Y (-E) \oplus \mathcal{O}_Y \to \varphi_{\ast} \mathcal{O}_X \bigl((H-f^{\ast}E)_{|_X} \bigr) \to R^1 f_{\ast} \mathcal{O}_{\mathcal{B}} \bigl(-(m-1)H -f^{\ast}(E+\Delta)\bigr) \to 0 .
\end{equation}
For the term  $R^1 f_{\ast} \mathcal{O}_{\mathcal{B}} \bigl( -(m-1)H -f^{\ast}(E+\Delta) \bigr)$, we find:
\[
 \begin{aligned}
  R^1 f_{\ast} \mathcal{O}_{\mathcal{B}} \bigl(-(m-1)H -f^{\ast}(E+\Delta) \bigr) & \cong R^1 f_{\ast} \mathcal{O}_{\mathcal{B}} \bigl(-(m-1)H \bigr) \otimes \mathcal{O}_Y (-(E+\Delta)) \\
  & \cong \bigl( \Sym^{m-3} (\mathcal{O}_Y \oplus \mathcal{O}_Y (E)) \bigr)^{\vee} \otimes \mathcal{O}_Y (-(2E+\Delta)) \\
  & \cong \mathcal{O}_Y (-2E-\Delta) \oplus \cdots \oplus \mathcal{O}_Y (-(m-1)E-\Delta) .
 \end{aligned}
\]
Computing $\Ext^1 \bigl( R^1 f_{\ast} \mathcal{O}_{\mathcal{B}} (-(m-1)H -f^{\ast}(E+\Delta)), \mathcal{O}_Y (-E) \oplus \mathcal{O}_Y \bigr)$, we obtain
\[
 \begin{aligned}
  \Ext^1 & \bigl( R^1 f_{\ast} \mathcal{O}_{\mathcal{B}} (-(m-1)H -f^{\ast}(E+\Delta)), \mathcal{O}_Y (-E) \oplus \mathcal{O}_Y \bigr) \\
  & \cong H^1 \bigl( Y, \left[ \mathcal{O}_Y (2E+\Delta) \oplus \cdots \oplus \mathcal{O}_Y ((m-1)E+\Delta) \right] \otimes \left[ \mathcal{O}_Y (-E) \oplus \mathcal{O}_Y \right] \bigr) \\
  & \cong H^1 \bigl( Y, \left[\mathcal{O}_Y (E) \oplus \oplus^2_1 \mathcal{O}_Y (2E) \oplus \cdots \oplus \oplus^2_1 \mathcal{O}_Y ((m-2)E) \oplus \mathcal{O}_Y \bigl( (m-1)E \bigr) \right]\otimes \mathcal{O}_Y (\Delta) \bigr) \\
  & = 0 ,
 \end{aligned}
\]
where the vanishing follows from the assumptions of the theorem. Therefore, the exact sequence~\eqref{ExSeq2} splits, yielding the desired decomposition.

This completes the proof of the theorem.
\end{proof}

\medskip

\section{The case of curves}\label{Sec3}

In this section, we specialize to the case where $Y$ is a smooth projective curve and $E$ is an effective, nonspecial divisor of degree $e$ on $Y$. In this situation, the bundle
\[
 \mathcal{B} = \mathbb{P} \bigl(\mathcal{O}_Y \oplus \mathcal{O}_Y(E) \bigr)
\]
is a decomposable ruled surface.
Before deriving the corollaries of \textnormal{Theorem~\ref{TheoremM}} in this setting, we recall some classical facts about decomposable ruled surfaces.

The surface $\mathcal{B}$ contains a section $\sigma_0$ corresponding to
\[
 \sigma_0 : \quad
 \mathcal{O}_Y \oplus \mathcal{O}_Y (E)
 \twoheadrightarrow \mathcal{O}_Y .
\]
Let $Y_0 := \sigma_0(Y)$. Then $Y_0$ is the \emph{section of minimal self-intersection}, satisfying $Y_0^2 = -e$.
The surface $\mathcal{B}$ also contains another section $\sigma_1$, associated with the surjection
\[
 \sigma_1 : \quad
 \mathcal{O}_Y \oplus \mathcal{O}_Y (E) \twoheadrightarrow \mathcal{O}_Y (E) .
\]
If we denote $Y_1 := \sigma_1(Y)$, then
\[
 Y_1 \sim Y_0 + f^{\ast}E \sim H ,
\]
where $H$ is the tautological divisor as in the general case. The intersection numbers are given by
\[
 Y_0 \cdot H = 0,
 \qquad
 H^2 = e .
\]

\begin{prop}\label{Sec3_Prop_Morph}
Suppose that $Y$ is a smooth projective curve of genus $\gamma$, $E$ is a very ample nonspecial divisor of degree $e$, and $\mathcal{B}$, $Y_0$, $Y_1$, and $H$ are as above.
\begin{enumerate}[label=(\roman*), leftmargin=*, font=\rmfamily]
 \item The linear series $|\mathcal{O}_{\mathcal{B}}(H)|$ is base-point-free and defines a morphism
 \[
   \Psi := \Psi_{|\mathcal{O}_{\mathcal{B}}(H)|} : \mathcal{B} \longrightarrow \mathbb{P}^R ,
 \]
 where $R = e - \gamma + 1$.

 \item The morphism $\Psi$ is an isomorphism away from $Y_0$, which it contracts to a point.

 \item Geometrically, the image $F := \Psi(\mathcal{B})$ is a cone in $\mathbb{P}^R$ over a smooth curve $Y_e \cong Y$ of degree $e$, embedded in $\mathbb{P}^{R-1}$. The lines in the ruling of $F$ are the images of the fibers $f^{-1}(q)$ for $q \in Y$.
\end{enumerate}
\end{prop}
\begin{proof}
The proof is obtained easily using  \cite[Ex.~V.2.11, p.385]{Hart77} and  \cite[Proposition~23]{GP2005}. For additional details the reader can refer to \cite[Lemma 2.2]{CIK25}.
\end{proof}

\medskip

In the above setting, let $P$ denote the vertex of the cone $F \subset \mathbb{P}^R$. We also note that in \textnormal{Corollary \ref{Coro1}} and \textnormal{Corollary \ref{Coro2}} that follow, it is assumed that the divisor $E$ on the curve $Y$ of genus $\gamma$ is nonspecial and very ample.

\begin{coro}\label{Coro1}
Let $X \subset \mathcal{B}$ be an irreducible smooth curve in the linear series $|mH|$, and let $X_{me} := \Psi(X) \subset F$ be the image of $X$ under $\Psi$. Then:
\begin{enumerate}[label=(\alph*), leftmargin=*, font=\rmfamily]
 \item $X_{me}$ is a smooth curve of degree $me$ and genus
 \[
   g = \binom{m}{2}e + m\gamma + 1 - m ,
 \]
 cut from $F$ by a degree $m$ hypersurface in $\mathbb{P}^{R}$;

 \item if $\phi : X_{me} \to Y_e$ is the morphism induced by projection from $P$ to the hyperplane in $\mathbb{P}^R$ containing $Y_e$, then
 \begin{equation}\label{Sec3Coro1eq}
  \phi_{\ast}\mathcal{O}_{X_{me}} \;\cong\;
  \mathcal{O}_{Y_e} \oplus \mathcal{O}_{Y_e}(-1) \oplus \cdots \oplus \mathcal{O}_{Y_e}\bigl(-(m-1)\bigr).
 \end{equation}
\end{enumerate}
\end{coro}
\begin{proof}
Since $H \cdot Y_0 = 0$, the morphism $\Psi$ maps each smooth curve in the linear system $|kH|$ on $\mathcal{B}$ (for $k \geq 1$) isomorphically onto a curve on $F$, cut out by a hypersurface of degree $k$ in $\mathbb{P}^R$. In particular, $Y_e \cong Y$ appears as a hyperplane section of $F$, and  $X_{me} \cong X$ as a hypersurface section of degree $m$. The formula for the genus then follows directly from the adjunction formula. Also, $\deg X_{me} = m H \cdot H = me$. This establishes \textnormal{(a)}.

For part \textnormal{(b)}, note that the morphism $\phi : X_{me} \to Y_e$ corresponds, via the diagram below, to the morphism $\varphi : X \to Y$ from \textnormal{Theorem~\ref{TheoremM}} with $\Delta = 0$:
\begin{equation}\label{Sec3Coro1CD}
    \begin{tikzcd}
      X \ar[r, "{\Psi_{|_{X}}}", "{\cong}"'] \ar[d, "{\varphi}"'] & X_{me} \ar[d, "{\phi}"] \\
      Y \cong Y_1 \ar[r, "{\Psi_{|_{Y_1}}}", "{\cong}"'] & Y_e
    \end{tikzcd}
\end{equation}
Moreover,
\[
 \bigl( \Psi_{|_{Y_1}} \bigr)^{\ast} \mathcal{O}_{Y_e}(1)
   \;\cong\; \mathcal{O}_{Y_1}\bigl(H_{|_{Y_1}}\bigr)
   \;\cong\; \bigl(f_{|_{Y_1}}\bigr)^{\ast}\mathcal{O}_{Y}(E).
\]
Thus, applying part \textnormal{(a)} of \textnormal{Theorem~\ref{TheoremM}} yields the decomposition \eqref{Sec3Coro1eq}, completing the proof.
\end{proof}

%%%%%%%%%%%%%%%%%%%%%%%%%%%%%%%%%%%%%%%%%%%%%%%

If $q \in Y$ is a point, then the fiber $f^{\ast}q = f^{-1}(q)$ of $f : \mathcal{B} \to Y$ meets the section $Y_0 \sim H - f^{\ast}E$ in a unique point, denoted
\[
 q_0 := Y_0 \cap f^{\ast}q .
\]
If $X \in |mH + f^{\ast}q|$ is a smooth curve, then
\begin{equation}\label{Sec3RestrY0}
 \mathcal{O}_X\big((H - f^{\ast}E)|_X \big) \;\cong\; \mathcal{O}_X\big(Y_0|_X \big) \;\cong\; \mathcal{O}_X(q_0).
\end{equation}
Moreover, for any $z \in Y$, the morphism $\Psi$ maps the fiber $f^{\ast}z = f^{-1}(z)$ onto a line in the ruling of $F$.

\begin{remark}\label{Rem_smoothC}
As noted in \cite[p.~226]{CG99}, the linear system $|mH + f^{\ast}q|$ contains a smooth element provided there exist reduced divisors
$\tilde{E} \in |E|$ and $G \in |mE+q|$ such that $q \in \tilde{E}$ and $\tilde{E} \cap G = \varnothing$.
Since $E$ is a nonspecial divisor on $Y$, the system $|mE+q|$ is base-point free, so $q$ is not a fixed point.
When $E$ is nonspecial and $h^0\bigl(Y, \mathcal{O}_Y(E)\bigr) \geq 2$, Bertini-type arguments guarantee that one can choose $\tilde{E} \sim E$ with $q \in \tilde{E}$ and a reduced divisor $G \sim mE+q$ disjoint from $\tilde{E}$.
Hence, whenever $E$ is a nonspecial divisor on $Y$ with $h^0\bigl(Y, \mathcal{O}_Y(E)\bigr) \geq 2$, the general element of $|mH + f^{\ast}q|$ is smooth and irreducible.
\end{remark}

%%%%%%%%%%%%%%%%%%%%%%%%%%%%%%%%%%%%%%%%%%%%%%%

\begin{coro}\label{Coro2}
Let $\Delta = q$ be a point on $Y$, and let $X \subset \mathcal{B}$ be a smooth curve in the linear series $|mH + f^{\ast} q|$.
Denote by $X_{me+1} := \Psi(X) \subset F$ the image of $X$ under $\Psi$.
Let $l_q$ be the line in the ruling of $F$ corresponding to the image of $f^{\ast} q$, and let $Q := l_q \cap Y_e$.
Then:
\begin{enumerate}[label=(\alph*), leftmargin=*, font=\rmfamily]
 \item $X_{me+1}$ is a smooth curve of degree $me+1$ and genus
 \[
   g = \binom{m}{2}e + m\gamma ,
 \]
 which passes through the vertex $P$ of $F$ with multiplicity one;

 \item if $\phi : X_{me+1} \to Y_e$ is the morphism induced by projection from $P$ onto the hyperplane in $\mathbb{P}^R$ containing $Y_e$, then:
 \begin{itemize}[label=$\bullet$, leftmargin=*, font=\rmfamily]
  \item
    \begin{equation}\label{Sec3Coro2eq1}
     \phi_{\ast}\mathcal{O}_{X_{me+1}} \;\cong\;
     \mathcal{O}_{Y_e} \oplus
     \Bigl(\mathcal{O}_{Y_e}(-1) \oplus \cdots \oplus \mathcal{O}_{Y_e}\bigl(-(m-1)\bigr)\Bigr)
     \otimes \mathcal{O}_{Y_e}(-Q) ,
    \end{equation}

  \item
    \begin{equation}\label{Sec3Coro2eq2}
     \phi_{\ast}\mathcal{O}_{X_{me+1}} (P) \;\cong\;
     \mathcal{O}_{Y_e} \oplus \mathcal{O}_{Y_e}(-1) \oplus
     \Bigl(\mathcal{O}_{Y_e}(-2) \oplus \cdots \oplus \mathcal{O}_{Y_e}\bigl(-(m-1)\bigr)\Bigr)
     \otimes \mathcal{O}_{Y_e}(-Q) .
    \end{equation}
 \end{itemize}
\end{enumerate}
\end{coro}
\begin{proof}
Part~\textnormal{(a)} follows by arguments analogous to those used in Corollary~\ref{Coro1}.
For the statement about the multiplicity at $P$, see also \cite[Prop.~3 and 4]{CIK24b}.

For part~\textnormal{(b)}, note that the projection from $P$ to the hyperplane containing $Y_e$ induces the morphism $\phi : X_{me+1} \to Y_e$.
This corresponds to the morphism $\varphi : X \to Y$ from \textnormal{Theorem~\ref{TheoremM}}, via a commutative diagram analogous to \eqref{Sec3Coro1CD}.
The claimed decompositions \eqref{Sec3Coro2eq1}–\eqref{Sec3Coro2eq2} then follow from \eqref{Sec3RestrY0} and \textnormal{Theorem~\ref{TheoremM}}.
\end{proof}

The previous corollaries allow us to derive the following result.

\begin{prop}\label{Sec3PropC_on_cones}
Let $Y_e \subset \mathbb{P}^{R-1}$ be a smooth curve of genus $\gamma \geq 2$ and degree $e \geq 2\gamma - 1$, with $R = e - \gamma + 1$. Let $F \subset \mathbb{P}^{R}$ be the non-degenerate cone over $Y_e$ with vertex $P$. Fix an integer $m \geq 2$, and let $X_d \subset F$ be a smooth curve of degree $d$ and genus $g$. Denote by
\[
 \phi : X_d \longrightarrow Y_e
\]
the morphism induced by projection from $P$ onto the hyperplane containing $Y_e$. Then:
\begin{enumerate}[label=(\alph*), leftmargin=*, font=\rmfamily]
 \item If $d = me$ and $g = \binom{m}{2}e + m\gamma + 1 - m$, then $X_d$ is cut out on $F$ by a hypersurface of degree $m$. Moreover, the direct image $\phi_{\ast}\mathcal{O}_{X_d}$ decomposes as in~\eqref{Sec3Coro1eq}.

 \item If $d = me+1$ and $g = \binom{m}{2}e + m\gamma$, then $X_d$ is algebraically equivalent to the intersection of $F$ with a hypersurface of degree $m$, together with a line from the ruling of $F$. The tangent line $\tau_P$ to $X_d$ at $P$ belongs to the ruling of $F$ and meets $X_d$ at $P$ with multiplicity two. Setting $Q := \tau_P \cap Y_e$, the pushforwards $\phi_{\ast}\mathcal{O}_{X_d}$ and $\phi_{\ast}\mathcal{O}_{X_d}(P)$ satisfy~\eqref{Sec3Coro2eq1} and~\eqref{Sec3Coro2eq2}, respectively.
\end{enumerate}
\end{prop}
\begin{proof}
Let $S$ be the blow-up of $F$ at its vertex $P$. Then $S$ can be identified with the decomposable ruled surface
\[
 \mathbb{P} \bigl( \mathcal{O}_{Y_e} \oplus \mathcal{O}_{Y_e}(1) \bigr)
\]
over $Y_e$. Let $f : S \to Y_e$ be the natural projection, and let $H$ denote the tautological divisor on $S$, so that
\[
 f_{\ast}\mathcal{O}_S(H) \;\cong\; \mathcal{O}_{Y_e} \oplus \mathcal{O}_{Y_e}(1).
\]
Since $F$ is the image of $S$ under the morphism $\Psi$ associated with the linear system $|H|$, it suffices to show that:
\begin{itemize}[leftmargin=*]
 \item in case \textnormal{(a)}, the curve $X_d$ is the image of a smooth curve $X \sim mH$ on $S$;
 \item in case \textnormal{(b)}, the curve $X_d$ is the image of a smooth curve $X \sim mH + f^{\ast}Q$ on $S$.
\end{itemize}

We treat case \textnormal{(a)}; case \textnormal{(b)} was established in \cite[Proposition~5]{CIK24b}.

Assume $d = me$ and $g = \binom{m}{2}e + m\gamma + 1 - m$ as in \textnormal{(a)}. Let $X \subset S$ be the strict transform of $X_d$, and let $Y_0 := \Psi^{-1}(P)$ denote the section of minimal self-intersection on $S$. Then $Y_0^2 = -e$, and
\[
 H \;\sim\; Y_0 + f^{\ast}\xi ,
\]
where $\xi$ is the hyperplane section divisor on $Y_e$, i.e. $\mathcal{O}_{Y_e}(\xi) \cong \mathcal{O}_{Y_e}(1)$.

Since $S$ is generated by $Y_0$ and pullbacks of divisors from $Y_e$, we may write
\[
 X \;\sim\; aY_0 + f^{\ast}\beta
\]
for some divisor $\beta$ on $Y_e$. From the degree condition
\[
 H \cdot X = \deg X_d = me ,
\]
we obtain $\deg \beta = me$.

Applying the adjunction formula gives
\[
\bigl(-2Y_0 + f^{\ast}\omega_{Y_e} - f^{\ast}\xi + aY_0 + f^{\ast}\beta \bigr)\cdot \bigl(aY_0 + f^{\ast}\beta \bigr) = 2g-2 = m(m-1)e + 2m(\gamma-1),
\]
which simplifies to
\begin{equation}\label{Prop6AdjForm}
 -a(a-2)e + me(a-2) + a(2\gamma - 2 - e + me) \;=\; m\bigl((m-1)e + 2(\gamma - 1)\bigr).
\end{equation}

Solving \eqref{Prop6AdjForm} for $a$, we find the roots
\[
 a = m \quad \text{or} \quad a = m+1 + \tfrac{2(\gamma-1)}{e}.
\]
Since $a$ must be integer and $e \geq 2\gamma-1$ by hypothesis, the second solution is excluded. Thus $a = m$, and hence
\[
 X \;\sim\; mY_0 + f^{\ast}\beta .
\]

Next, observe
\[
 X \cdot Y_0 = (mY_0 + f^{\ast}\beta)\cdot Y_0 = -me + \deg\beta = 0,
\]
so the curves $X$ and $Y_0$ are disjoint. Consequently,
\[
 (f^{\ast}\beta)_{|Y_0} \;\sim\; -m(Y_0|_{Y_0}).
\]
Identifying $Y_0 \cong Y_e$ via the isomorphism $j := {f}_{|_{Y_0}} : Y_0 \to Y_e$, we have
\[
 Y_0|_{Y_0} \;\sim\; -j^{\ast}\xi \quad \text{and} \quad (f^{\ast}\beta)_{|Y_0} \;\sim\; j^{\ast}\beta .
\]
It follows that $j^{\ast}\beta \sim m j^{\ast}\xi$, hence $\beta \sim m\xi$ on $Y_e$. Therefore,
\[
 X \;\sim\; m (Y_0 + f^{\ast}\xi) \;\sim\; mH .
\]

The claim in \textnormal{(a)} now follows from \textnormal{Corollary~\ref{Coro1}}.
\end{proof}

\begin{coro}\label{Sec3PropC_on_conesCoro}
In the setting of \textnormal{Proposition \ref{Sec3PropC_on_cones}}, let $X_d \subset F$ be a smooth curve of degree $d$ such that the morphism $\phi : X_d \to Y_e$ is an $m\!:\!1$ cover. Then:
\begin{enumerate}[label=(\alph*), leftmargin=*, font=\rmfamily]
 \item If $X_d$ does not pass through $P$, then $d = me$ and $g = \binom{m}{2}e + m\gamma + 1 - m$. In particular, case \textnormal{(a)} of \textnormal{Proposition \ref{Sec3PropC_on_cones}} applies.

 \item If $X_d$ passes through $P$, then $d = me+1$ and $g = \binom{m}{2}e + m\gamma$. In particular, case \textnormal{(b)} of \textnormal{Proposition \ref{Sec3PropC_on_cones}} applies.
\end{enumerate}
\end{coro}
\begin{proof}
Let $S$, $f$, $X$, $H$, $Y_0$, $\xi$, and $\Psi$ be as in the proof of \textnormal{Proposition \ref{Sec3PropC_on_cones}}. Write
\[
 X \sim a Y_0 + f^{\ast}\beta ,
\]
for some integer $a$ and some divisor $\beta$ on $Y_e$. Since $\Psi$ maps fibers $f^{\ast}z$ (for $z \in Y_e$) to lines in the ruling of $F$, and $\phi$ is the projection from $P$ onto the hyperplane containing $Y_e$, the assumption that $\phi$ is an $m\!:\!1$ cover implies
\[
 m = X \cdot f^{\ast}z = a .
\]

\smallskip
\noindent\textbf{(a)} If $X_d$ does not contain $P$, then
\[
 Y_0 \cdot (mY_0 + f^{\ast}\beta) = 0 ,
\]
hence $Y_0 \cdot f^{\ast}\beta = me$. By the adjunction formula we obtain
\[
 2g-2 = \bigl(-2Y_0 + f^{\ast}\omega_{Y_e} - f^{\ast}\xi + mY_0 + f^{\ast}\beta \bigr)\cdot (mY_0 + f^{\ast}\beta) = m(m-1)e + 2m(\gamma-1).
\]
Thus $g = \binom{m}{2}e + m\gamma + 1 - m$. The degree $d$ of $X_d$ is
\[
 d = X \cdot H = (mY_0 + f^{\ast}\beta)\cdot (Y_0 + f^{\ast}\xi) = me.
\]

\smallskip
\noindent\textbf{(b)} If $X_d$ contains $P$, then
\[
 Y_0 \cdot (mY_0 + f^{\ast}\beta) = 1 ,
\]
since $\Psi$ contracts $Y_0$ to $P$. Hence $Y_0 \cdot f^{\ast}\beta = me+1$. By adjunction,
\[
 2g-2 = m(m-1)e + 2m\gamma - 2 ,
\]
so $g = \binom{m}{2}e + m\gamma$. The degree $d$ of $X_d$ is
\[
 d = (mY_0 + f^{\ast}\beta)\cdot (Y_0 + f^{\ast}\xi) = me+1.
\]
\end{proof}

Finally, we include one more statement, concerning plane curves, which is derived easily using similar arguments as above.

\begin{prop}\label{Sec3_Prop_PlaneC}
Let $X_m \subset \mathbb{P}^2$ be a smooth plane curve of degree $m \geq 2$, let $L \cong \mathbb{P}^1$ be a line in the plane, and let $P \in \mathbb{P}^2 \setminus L$ be a point not lying on $L$. Consider the morphism
\[
 \phi : X_m \longrightarrow L
\]
induced by projection from $P$ onto $L$. Then:

\begin{enumerate}[label=(\alph*), leftmargin=*, font=\rmfamily]
 \item If $P \notin X_m$, then $\phi$ is an $m\!:\!1$ covering, and
 \begin{equation*}\label{Prop_PlaneC_a}
  \phi_{\ast}\mathcal{O}_{X_m} \;\cong\;
  \mathcal{O}_{\mathbb{P}^1} \oplus \mathcal{O}_{\mathbb{P}^1}(-1) \oplus \cdots \oplus \mathcal{O}_{\mathbb{P}^1}\bigl(-(m-1)\bigr),
 \end{equation*}
 where we identify $\mathcal{O}_L$ with $\mathcal{O}_{\mathbb{P}^1}$.

 \item If $P \in X_m$, $m \geq 3$, and the tangent line to $X_m$ at $P$ meets $X_m$ with multiplicity two, then $\phi$ is an $(m-1)\!:\!1$ covering, and:
 \begin{itemize}[label=$\bullet$, leftmargin=*, font=\rmfamily]
  \item
    \begin{equation*}\label{Prop_PlaneC_b1}
     \phi_{\ast}\mathcal{O}_{X_m} \;\cong\;
     \mathcal{O}_{\mathbb{P}^1} \oplus \mathcal{O}_{\mathbb{P}^1}(-2) \oplus \mathcal{O}_{\mathbb{P}^1}(-3) \oplus \cdots \oplus \mathcal{O}_{\mathbb{P}^1}\bigl(-(m-1)\bigr),
    \end{equation*}

  \item
    \begin{equation*}\label{Prop_PlaneC_b2}
     \phi_{\ast}\mathcal{O}_{X_m}(P) \;\cong\;
     \mathcal{O}_{\mathbb{P}^1} \oplus \mathcal{O}_{\mathbb{P}^1}(-1) \oplus \mathcal{O}_{\mathbb{P}^1}(-3) \oplus \cdots \oplus \mathcal{O}_{\mathbb{P}^1}\bigl(-(m-1)\bigr).
    \end{equation*}
 \end{itemize}
\end{enumerate}
\end{prop}
\begin{proof}
Let $S$ be the blow-up of $\mathbb{P}^2$ at the point $P$. Then $S \cong \mathbb{F}_1 = \mathbb{P} \bigl( \mathcal{O}_{\mathbb{P}^1} \oplus \mathcal{O}_{\mathbb{P}^1}(1) \bigr)$, the first Hirzebruch surface, which can be viewed as the projectivization $S \cong \mathbb{P} \bigl( \mathcal{O}_{L} \oplus \mathcal{O}_{L}(1) \bigr)$ with projection map $f : S \to L$. Let $H$ denote the tautological divisor on $S$, and let $X \subset S$ be the strict (proper) transform of $X_m$. Then:
\begin{itemize}[leftmargin=*]
 \item in case \textnormal{(a)} we have $X \sim mH$,
 \item in case \textnormal{(b)} we have $X \sim (m-1)H + f^{\ast}Q$, where $Q$ is the point where the tangent line to $X_m$ at $P$ meets $L$.
\end{itemize}
The assertions now follow from \textnormal{Corollary~\ref{Coro1}} and \textnormal{Corollary~\ref{Coro2}}, respectively.
\end{proof}

\begin{remark}
In all of the above cases, the decomposition of the Tschirnhausen module is of \emph{Veronese type}, namely
\[
 \mathcal{E}^{\vee} \;\cong\; \bigl(\mathcal{O}_Y \oplus L \oplus \cdots \oplus L^{m-2}\bigr) \otimes M
\]
for some line bundles $L$ and $M$ on $Y$.
It would be interesting to find examples of such decompositions where $X$ is not an $m$-section of a decomposable $\mathbb{P}^1$-bundle over $Y$.
\end{remark}

\bigskip

%---------------Bibliography--------------------
%\bibliographystyle{plainurl}
%\bibliographystyle{abbrvurl}
\bibliographystyle{IEEEtran}

\end{document}